\documentclass{article}

\usepackage{amsmath}
\usepackage{amssymb}

\usepackage{amscd}

\usepackage{amsthm}

\newtheorem{theorem}{Theorem}
\newtheorem{conclusion}[theorem]{Conclusion}

\newtheorem{proposition}[theorem]{Proposition}
\newtheorem{example}[theorem]{Example}

\begin{document}

\author{Monika Maj\footnote{Kazimierz Pu\l aski University of Technology and Humanities
in Radom}  %
\and Zbigniew Pasternak-Winiarski\footnote{Faculty of Mathematics
and Information Science,  Warsaw University of Technology, Ul.
Koszykowa 75, 00-662 Warsaw, POLAND, Institute of Mathematics,
University of Bia\l ystok Akademicka 2, 15-267 Bia\l ystok,
Poland, e-mail: Z.Pasternak-Winiarski@mini.pw.edu.pl }}
\title{Zeros of densities and decomposition problem for multidimensional entire characteristic functions
of order 2}

\maketitle

\begin{abstract}
We consider the entire characteristic functions of order 2 and we
prove some decomposition theorems in a multidimensional case. We
show that the lack of zeros of the density function is a necessary
but not a sufficient (as in the one-dimensional case) condition
for a characteristic function to be decomposable. We also find
some simple sufficient conditions.
\bigskip


\bigskip
\end{abstract}

\section{Preliminaries}

Entire characteristic functions of order 2 with finite number of
zeros were considered by Lukacs. In \cite{lukacz} and
\cite{lukacz3} he presents theorems related to characteristic
functions of the form
\begin{equation}
\varphi (t)=P(t)\exp (A(t)), \qquad t\in\mathbb{R}
\label{fch}
\end{equation}%
where $P$ and $A$ are polynomials and $A$ is of order 2. $P$ is a polynomial of an even degree and has a form
\begin{equation*}
P(t)=\prod_{j=1}^{d}(1-\dfrac{t}{\xi_{j}})(1+\dfrac{t}{\overline{\xi_{j}}}),
\end{equation*}
where $\xi_{j}$ and $\overline{\xi_{j}}$ are zeros of $P$. Lukacs
has solved decomposition problem for characteristic function of
the form (1) where

\begin{equation}  \label{I2}
\varphi(t)=P(t)exp{ \left[ -\dfrac{\sigma^{2}t^{2}}{2} \right]}, \qquad t\in\mathbb{R},
\end{equation}

Then the density function corresponding to the characteristic
function (\ref{I2}) has the form

\begin{equation*}
f(x)=\frac{1}{\sigma \sqrt{2\pi }}Q(x)\exp \left[ -\frac{x^{2}}{2\sigma ^{2}}%
\right] , \qquad x\in\mathbb{R}, \label{zero1}
\end{equation*}

where the polynomial $Q$ can be written as

\[
Q(x)=\sum\limits_{k=0}^{2n}\left( -1\right) ^{k}\lambda _{k}\sigma
^{-k}H_{k}\left( \frac{x}{\sigma }\right) .
\]%

Here every $H_{k}$ is the Hermite polynomial of order $k$ and
$\lambda_k\in\mathbb{R}$ for $k=1,2,...,2n.$ It is also clear that
the polynomial $Q$ must be non-negative for all $x\in\mathbb{R}$.
In this paper we will call $Q$ the polynomial {\it associated}
with the characteristic function $\varphi$. We have two
possibilities for characteristic function of the form (\ref{I2}):
either $\varphi$ is indecomposable or it admits a decomposition

\begin{equation}
\varphi (t)=\varphi _{1}(t)\varphi _{2}(t), \qquad x\in\mathbb{R},
 \label{fch3}
\end{equation}%
where $\varphi _{1}$ and $\varphi _{2}$ are non-trivial
characteristic functions. There are again two possibilities in the
case where $\varphi$ of the form (\ref{fch3}) is decomposable.
From Pluci\'{n}ska (and Lukacs) theorem (see \cite{pluc4}) we have

(a) $\varphi $ has a normal factor \mbox{$\varphi _{1}(t)=\exp \left[ -%
\frac{\sigma _{1}^{2}t^{2}}{2}\right] $} and then\linebreak[4] $\varphi _{2}(t)=P(t)\exp %
\left[ -\frac{\sigma _{2}^{2}t^{2}}{2}\right] ,$ where $\sigma
_{1}^{2}+\sigma _{2}^{2}=\sigma ^{2}$ or

(b) $\varphi $ has factors of the form (\ref{I2}) i.e.%
\[
\varphi _{j}(t)=P_{j}(t)\exp \left[ -\frac{\sigma _{j}^{2}t^{2}}{2}\right] ,%
\text{ \ }j=1,2,
\]%
where $\sigma _{1}^{2}+\sigma _{2}^{2}=\sigma ^{2}$,
$P(t)=P_{1}(t)P_{2}(t),\ degP_1>0$ and $degP_2>0$.

The following theorems can be find in  \cite{lukacz3}.

\begin{theorem}
Suppose that the characteristic function $\varphi$ of the form
(\ref{I2}) admits a non-trivial decomposition.
 Then its associated polynomial $Q$ has no real zeros (see \cite{lukacz3}
 th.7.3.1).
\end{theorem}

\begin{theorem}
Let $\varphi $ be the entire characteristic function of the form
(\ref{I2}) and suppose that the polynomial associated with
$\varphi $ has no real zeros. Then $\varphi $ has a normal factor
(see (a) above and \cite{lukacz3} th.7.3.2).
\end{theorem}

From the above theorems we see that the possibility of
decomposition of a characteristic function (of the considered
type) is equivalent to the statement that the associated
polynomial has no real zeros. The aim of this paper is to prove
theorems related to the decomposition of multidimensional
characteristic functions. We will consider multidimensional
polynomial-normal distribution of the form

\[
f_{2l}(\mathbf{x})=\frac{\sqrt{\det \mathbf{A}}}{\left( 2\pi \right) ^{\frac{d}{2}}}%
p_{2l}(\mathbf{x})\exp \left( -\frac{1}{2}\left( \mathbf{x}-\mathbf{b}%
\right)^T \mathbf{A}\left( \mathbf{x}-\mathbf{b}\right)  \right) , \qquad x\in\mathbb{R}^{d},
\]%
where $p_{2l}$ is non-negative polynomial of order $2l$,
$\mathbf{b}\in\mathbb{R}^d$ and $\mathbf{A}$ is a nondegenerate
positive $d\times{d}$ matrix. As we know it from the theory of
Fourier transformation (see also Lukacs \cite{lukacz3}) the
characteristic function of a polynomial-normal distribution is a
product of some polynomial and the characteristic function of the
normal distribution defined by the same matrix $\mathbf{A}$ and
the same vector $\mathbf{b}$. We will prove the following theorem

\begin{theorem}
 Let the characteristic function  $\varphi $
of d-dimensional polynomial-normal distribution has a non-trivial decomposition
 $\varphi =\varphi _{1}\varphi _{2},$  where $\varphi_{1},\varphi _{2}$  are characteristic functions.
 Then its associated polynomial $Q$  has no real zeros.
\end{theorem}

\begin{proof}
 Let $f_{2l_{1}}$ and $f_{2l_{2}}$
be the densities corresponding to characteristic functions
 $\varphi _{1}$ and $\varphi _{2}$
respectively.Then by \cite{PWM}, Theorem 2, we have
\[
f_{2l_{1}}(\mathbf{x})=\frac{\sqrt{\det \mathbf{A}_{1}}}{\left( 2\pi \right) ^{\frac{d%
}{2}}}p_{2l_{1}}(\mathbf{x})\exp \left( -\frac{1}{2}\left( \mathbf{x}-%
\mathbf{b}\right)^T \mathbf{A}_{1}\left( \mathbf{x}-\mathbf{b}\right)  \right) ,
\]%
\begin{equation}
f_{2l_{2}}(\mathbf{x})=\frac{\sqrt{\det \mathbf{A}_{2}}}{\left( 2\pi \right) ^{\frac{d%
}{2}}}p_{2l_{2}}(\mathbf{x})\exp \left( -\frac{1}{2}\left( \mathbf{x}-%
\mathbf{b}_{2}\right)^T \mathbf{A}_{2}\left( \mathbf{x}-\mathbf{b}_{2}\right) \right) ,  
\label{w}
\end{equation}%
where $p_{2l_{1}},$ $p_{2l_{2}}$ are non-negative polynomials
determinated by the zeros of $\varphi _{1}$ and~$\varphi_{2}$
respectively, $\mathbf{A_{1}}$ and $\mathbf{A_{2}}$ are
non-degenerate, positive defined $d\times{d}$ matrixes and
$\mathbf{b_{1}},\mathbf{b_{2}\in\mathbb{R}^{d}}$. From (\ref{w}),
from the equality $\varphi (\mathbf{t})=\varphi
_{1}(\mathbf{t})\varphi _{2}(\mathbf{t})$
and from Borel theorem for Fourier transformation we have%
\[
f_{2l}(\mathbf{x})=\int\limits_{\mathbb{R}^{d}}f_{2l_{2}}(\mathbf{x-y}%
)f_{2l_{1}}(\mathbf{y})d\mathbf{y=}
\]%
\[
=\frac{\sqrt{\det \mathbf{A}_{1}\det \mathbf{A}_{2}}}{\left( 2\pi \right) ^{d}}\int\limits_{%
\mathbb{R}^{d}}p_{2l_{2}}(\mathbf{x-y})p_{2l_{1}}(\mathbf{y})\exp \left( -%
\frac{1}{2}\left( \mathbf{x-y}-\mathbf{b}_{2}\right)^T \mathbf{A}_{2}\left( \mathbf{x}-%
\mathbf{y}-\mathbf{b}_{2}\right)  \right) \times
\]%
\[
\times \exp \left( -\frac{1}{2}\left( \mathbf{y}-\mathbf{b}_{1}\right)^T
\mathbf{A}_{1}\left( \mathbf{y}-\mathbf{b}_{1}\right)  \right) d\mathbf{y,}  \qquad x\in\mathbb{R}^{d}.
\]%
 (see Maurin \cite{Mau}).

Let us assume that the polynomial $Q$ takes value zero at a point $\mathbf{x_{0}}\in\mathbb{R}^{d}$.

Then $f_{2l}(\mathbf{x}_{0})=0,$ so%
\[
\frac{\sqrt{\det \mathbf{A}_{1}\det \mathbf{A}_{2}}}{\left( 2\pi \right) ^{d}}\int\limits_{%
\mathbb{R}^{d}}p_{2l_{2}}(\mathbf{x}_{0}\mathbf{-y})p_{2l_{1}}(\mathbf{y}%
)\exp \left( -\frac{1}{2}\left( \mathbf{x}_{0}\mathbf{-y}-\mathbf{b}%
_{2}\right)^T \mathbf{A}_{2}\left( \mathbf{x}_{0}-\mathbf{y}-\mathbf{b}_{2}\right)
 \right) \times
\]%
\[
\times \exp \left( -\frac{1}{2}\left( \mathbf{y}-\mathbf{b}_{1}\right)^T
\mathbf{A}_{1}\left( \mathbf{y}-\mathbf{b}_{1}\right)  \right) d\mathbf{y=}0.
\]%
The above equality holds only in case where the integrand function
is equal zero, because it is non-negative and continuous. Then we
have
\[
\text{\ }\bigwedge\limits_{\mathbf{y}\in \mathbb{R}^{d}}p_{2l_{2}}(\mathbf{x}_{0}%
\mathbf{-y})p_{2l_{1}}(\mathbf{y})=0.
\]%
So%
\begin{equation}
\text{\ }\bigwedge\limits_{\mathbf{y}\in \mathbb{R}^{d}}p_{2l_{2}}(\mathbf{x}_{0}%
\mathbf{-y})=0\text{ }\vee p_{2l_{1}}(\mathbf{y})=0. 
\label{war}
\end{equation}%
 Since the polynomials $p_{2l_{1}}$ and $p_{2l_{2}}$ are not equal zero, the sets
\[
\left\{ \mathbf{y}\in \mathbb{R}^{d}:p_{2l_{2}}(\mathbf{x-y})=0\right\} \text{ and }%
\left\{ \mathbf{y}\in \mathbb{R}^{d}:p_{2l_{1}}(\mathbf{y})=0\right\}
\]%
are closed and their interiors in $\mathbb{R}^d$ are empty. Then
their sum has the empty interior (the trivial case of Baire
theorem) and the condition (\ref{war}) is not fulfilled. Hence one
of the polynomials $p_{2l_{1}}$ or $p_{2l_{2}}$ is equal zero -
contradiction.

Finally, the polynomial $Q$ associated with the given
characteristic function has no real zeros.
\end{proof}

Now we present an example which shows that Theorem 2 is not true
in multidimensional case.

\begin{example} {\rm
Let $\left( X_{1},X_{2}\right) $ be the 2-dimensional random variable
with the density%
\[
f(x_{1},x_{2})=\frac{1}{6\pi }\left[ \left( x_{1}x_{2}-1\right)
^{2}+x_{2}^{2}\right] \exp \left\{ -\frac{1}{2}\left(
x_{1}^{2}+x_{2}^{2}\right) \right\} =
\]%
\[
=\frac{1}{2\pi }p_{4}(x_{1},x_{2})\exp \left\{ -\frac{1}{2}\left(
x_{1}^{2}+x_{2}^{2}\right) \right\} .
\]%
Then the characteristic function has the form%
\begin{equation}
\varphi (t_{1},t_{2})=\frac{1}{3}\left[
t_{1}^{2}t_{2}^{2}+2t_{1}t_{2}-2t_{2}^{2}-t_{1}^{2}+3\right] \exp \left\{ -%
\frac{1}{2}\left( t_{1}^{2}+t_{2}^{2}\right) \right\} . 
\label{fchar}
\end{equation}%

By \cite{PWM}, Theorem 2 we know that if $\varphi$ has non-trivial
decomposition $\varphi=\varphi_{1}\varphi_{2}$ then $\varphi_{1}$
and $\varphi_{2}$ are characteristic functions of
polynomial-normal distributions. Let us first prove that in this
case it is a product of characteristic function of
polynomial-normal distribution and characteristic function of
normal distribution. It is a consequence of the fact that the
polynomial
\[ P(t_{1},t_{2})=t_{1}^{2}t_{2}^{2}+2t_{1}t_{2}-2t_{2}^{2}-t_{1}^{2}+3 \]
can not be presented as a product of two polynomials of degree $2.$

In fact let us write the polynomial $P$ as%
\[
W(t_{1})=t_{1}^{2}\left( t_{2}^{2}-1\right) +2t_{1}t_{2}-2t_{2}^{2}+3,
\]%
where for fixed $t_{2}$ it is a polynomial of degree 2 in $%
t_{1}.$ Then%
\[
\Delta =4\left( 2t_{2}^{4}-4t_{2}^{2}+3\right) >0\quad \text{for}\quad t_{2}\in \mathbb{R}.
\]%
So the roots have the form
\begin{equation}
t_{1,2}=\frac{-t_{2}\pm \sqrt{2t_{2}^{4}-4t_{2}^{2}+3}}{t_{2}^{2}-1}.
\label{6pier}
\end{equation}%
If the polynomial $P$ is decomposable %
there are two possibilities

\textbf{1. \ }
$P(t_{1},t_{2})=Q_{1}(t_{1},t_{2})Q_{2}(t_{1},t_{2}),$ where
\[
Q_{j}(t_{1},t_{2})=c_{j}t_{1}t_{2}+d_{j}t_{1}+e_{j}t_{2}+g_{j},\text{ \ }%
j=1,2.
\]%
The polynomial $P$ is equal zero in $\left(
t_{1},t_{2}\right) $ iff one of polynomial $Q_{j}$ ( or both of them) is equal zero. Let%
\[
c_{j}t_{1}t_{2}+d_{j}t_{1}+e_{j}t_{2}+g_{j}=0.
\]%
Then for a fixed $t_{2}$ we have
\begin{equation}
t_{1}=-\frac{e_{j}t_{2}+g_{j}}{c_{j}t_{2}+d_{j}}. \label{3.24} 
\end{equation}%

Formula (\ref{3.24}) is not of the form (\ref{6pier}). Indeed
homography function has only one pole, but roots in (\ref{6pier})
are functions with two poles
 $t_{2}=\pm 1.$

\textbf{2. } Let now %
\[
P(t_{1},t_{2})=Q_{1}(t_{1},t_{2})Q_{2}(t_{1},t_{2}),
\]
where
\[
Q_{1}(t_{1},t_{2})=a_{1}t_{1}^{2}+c_{1}t_{1}t_{2}+d_{1}t_{1}+e_{1}t_{2}+g_{1}
\]%
and%
\[
Q_{2}(t_{1},t_{2})=b_{2}t_{2}^{2}+c_{2}t_{1}t_{2}+d_{2}t_{1}+e_{2}t_{2}+g_{2}.
\]%
For fixed $t_{2}$ and $Q_{1}(t_{1},t_{2})=0$ we have
\[
\Delta
=c_{1}^{2}t_{2}^{2}+2c_{1}d_{1}t_{2}+d_{1}^{2}-4a_{1}e_{1}t_{2}-4a_{1}g_{1},
\]%
\begin{equation}
t_{1}=\frac{-c_{1}t_{2}-d_{1}\pm \sqrt{%
c_{1}^{2}t_{2}^{2}+2c_{1}d_{1}t_{2}+d_{1}^{2}-4a_{1}e_{1}t_{2}-4a_{1}g_{1}}}{%
2a_{1}}. \label{3.25} 
\end{equation}%
When $Q_{2}(t_{1},t_{2})=0$ then%
\begin{equation}
t_{1}=-\frac{b_{2}t_{2}^{2}+e_{2}t_{2}+g_{2}}{c_{2}t_{2}+d_{2}}. \label{3.26} 
\end{equation}%
Functions (\ref{3.25}) and (\ref{3.26}) have only one pole whereas
the function (\ref{6pier}) has two poles. Then the polynomial $P$
is indecomposable. We can also obtain this result  writing $P$ in
the form

\begin{small}\begin{eqnarray*}
P(t_{1},t_{2}) =\left(
a_{1}t_{1}^{2}+b_{1}t_{2}^{2}+c_{1}t_{1}t_{2}+d_{1}t_{1}+e_{1}t_{2}+g_{1}%
\right)    \left(
a_{2}t_{1}^{2}+b_{2}t_{2}^{2}+c_{2}t_{1}t_{2}+d_{2}t_{1}+e_{2}t_{2}+g_{2}%
\right),
\end{eqnarray*}\end{small}%
and comparing the coefficients in successive powers of variables
 $t_{1},t_{2}.$

Hence, if the characteristic function of the form (\ref{fchar}) is
decomposable, then it is a product of the characteristic function
of a polynomial-normal distribution and the characteristic
function of a normal distribution. Let
\begin{align*}
\varphi (t_{1},t_{2})=& \frac{1}{3}\left[
t_{1}^{2}t_{2}^{2}+2t_{1}t_{2}-2t_{2}^{2}-t_{1}^{2}+3\right] \exp \left\{ -%
\frac{1}{2}\left( a_{11}t_{1}^{2}+2a_{12}t_{1}t_{2}+a_{22}t_{2}^{2}\right)
\right\} \times
\\
& \times \exp \left\{ -\frac{1}{2}\left( t_{1}^{2}\left(
1-a_{11}\right) -2a_{12}t_{1}t_{2}+\left( 1-a_{22}\right)
t_{2}^{2}\right) \right\},
\end{align*}
be such a decomposition. We will show that it is not possible
because the function
\begin{equation}
\varphi _{1}(t_{1},t_{2})=\frac{1}{3}\left[
t_{1}^{2}t_{2}^{2}+2t_{1}t_{2}-2t_{2}^{2}-t_{1}^{2}+3\right] \exp \left\{ -%
\frac{1}{2}\left( a_{11}t_{1}^{2}+2a_{12}t_{1}t_{2}+a_{22}t_{2}^{2}\right)
\right\}, \label{3.27} 
\end{equation}
where
\[
  a_{11},a_{22}>0, \quad a_{11}a_{22}-a_{12}^2>0, \quad 1-a_{11},1-a_{22}>0, \quad (1-a_{11})(1-a_{22})-a_{12}^2>0,
\]
could not be a characteristic function of polynomial-normal distribution.

The density function for characteristic function of the form
(\ref{3.27}) is

\[
f_{1}(x_{1},x_{2})=\frac{1}{2\pi }\iint\limits_{\mathbb{R}^{2}}\frac{1}{3}%
\left[ t_{1}^{2}t_{2}^{2}+2t_{1}t_{2}-2t_{2}^{2}-t_{1}^{2}+3\right] \times
\]%
\[
\times \exp \left\{ -\frac{1}{2}\left(
a_{11}t_{1}^{2}+2a_{12}t_{1}t_{2}+a_{22}t_{2}^{2}\right) \right\} \exp
\left\{ -it_{1}x_{1}-it_{2}x_{2}\right\} dt_{2}dt_{1}=
\]%
\[
=\frac{1}{2\pi \sqrt{a_{11}a_{22}-a_{12}^{2}}}\frac{1}{3}\left[ \frac{%
a_{12}^{2}}{\left( a_{11}a_{22}-a_{12}^{2}\right) ^{2}}X^{4}-\frac{2a_{12}}{%
\left( a_{11}a_{22}-a_{12}^{2}\right) ^{\frac{3}{2}}}X^{3}Y+\right.
\]%
\[
+\frac{1}{a_{11}a_{22}-a_{12}^{2}}X^{2}Y^{2}+\left( \frac{6a_{12}}{\left(
a_{11}a_{22}-a_{12}^{2}\right) ^{\frac{3}{2}}}-\frac{4a_{12}}{a_{22}\sqrt{%
a_{11}a_{22}-a_{12}^{2}}}-\frac{2}{\sqrt{a_{11}a_{22}-a_{12}^{2}}}\right) XY+
\]%
\[
+\left( a_{22}+2a_{12}+\frac{a_{12}^{2}}{a_{22}}-1-\frac{6a_{12}^{2}}{%
a_{11}a_{22}-a_{12}^{2}}\right) \frac{X^{2}}{a_{11}a_{22}-a_{12}^{2}}+\left(
\frac{2}{a_{22}}-\frac{1}{a_{11}a_{22}-a_{12}^{2}}\right) Y^{2}+
\]%
\[
\left. +3-\frac{2}{a_{22}}+\frac{3a_{12}^{2}}{a_{11}a_{22}-a_{12}^{2}}+\frac{%
a_{22}}{a_{11}a_{22}-a_{12}^{2}}\left( -1-\frac{2a_{12}}{a_{22}}-\frac{%
a_{12}^{2}}{a_{22}^{2}}+\frac{1}{a_{22}}\right) \right] \times
\]%
\[
\times \exp \left\{ -\frac{1}{2}\left( X^{2}+Y^{2}\right) \right\} =\frac{1}{%
2\pi \sqrt{a_{11}a_{22}-a_{12}^{2}}}\frac{1}{3}\widetilde{Q}(X,Y)\exp
\left\{ -\frac{1}{2}\left( X^{2}+Y^{2}\right) \right\} ,
\]%
where%
\[
X=\frac{x_{1}-\frac{a_{12}}{a_{22}}x_{2}}{\sqrt{a_{11}-\frac{a_{12}^{2}}{%
a_{22}}}},\qquad Y=\frac{x_{2}}{\sqrt{a_{22}}}.
\]%
Let us see that for $a_{12}=0$ the coefficient of $X^4$ disappear
and the coefficient of  $X^2$ is negative. Then for $Y=0$ and for
$X$ large enough the polynomial $\widetilde{Q}$ has negative
values at points $(X,0)$ -- contradiction.

Let us consider the case when $a_{12}\neq 0.$
Now we prove that the polynomial in above density function is non-positive.
First suppose that $a_{12}>0$ and substitute %
\[
T=\sqrt{\frac{a_{12}}{a_{11}a_{22}-a_{12}^{2}}}X,
\]
 $a_{12}>0.$
Then
\[
\widetilde{Q}(X,Y)=T^{4}-\frac{2}{\sqrt{a_{12}}}T^{3}Y+\frac{T^{2}Y^{2}}{%
a_{12}}+\left( \frac{6\sqrt{a_{12}}}{a_{11}a_{22}-a_{12}^{2}}-\frac{4\sqrt{%
a_{12}}}{a_{22}}-\frac{a_{11}}{a_{22}\sqrt{a_{12}}}\right) TY+
\]%
\[
+\left( \frac{a_{22}}{a_{12}}+2+\frac{a_{12}}{a_{22}}-\frac{1}{a_{12}}-\frac{%
6a_{12}}{a_{11}a_{22}-a_{12}^{2}}\right) T^{2}+\left( \frac{2}{a_{22}}-\frac{%
1}{a_{11}a_{22}-a_{12}^{2}}\right) Y^{2}+
\]%
\[
+3-\frac{2}{a_{22}}+\frac{3a_{12}^{2}}{a_{11}a_{22}-a_{12}^{2}}+\frac{a_{22}%
}{a_{11}a_{22}-a_{12}^{2}}\left( -1-\frac{2a_{12}}{a_{22}}-\frac{a_{12}^{2}}{%
a_{22}^{2}}+\frac{1}{a_{22}}\right) =:Q_1(T,Y).
\]%
Let us denote
\[
\widetilde{p}_{4}(x_{1},x_{2})=\left( x_{1}x_{2}-1\right) ^{2}+x_{2}^{2}.
\]%
Thus
\[
Q_1(T,Y)=\widetilde{p}_{4}\left( T,T-\frac{Y}{\sqrt{a_{12}}}%
\right) +\left( 3+\frac{a_{22}}{a_{12}}+\frac{a_{12}}{a_{22}}-\frac{1}{a_{12}%
}-\frac{6a_{12}}{a_{11}a_{22}-a_{12}^{2}}\right) T^{2}+
\]%
\[
+\left( \frac{2}{a_{22}}-\frac{1}{a_{11}a_{22}-a_{12}^{2}}-\frac{1}{a_{12}}%
\right) Y^{2}+\left( \frac{6\sqrt{a_{12}}}{a_{11}a_{22}-a_{12}^{2}}-\frac{2}{%
\sqrt{a_{12}}}-\frac{4\sqrt{a_{12}}}{a_{22}}\right) TY+
\]%
\[
+2-\frac{2}{a_{22}}+\frac{3a_{12}^{2}}{a_{11}a_{22}-a_{12}^{2}}+\frac{a_{22}%
}{a_{11}a_{22}-a_{12}^{2}}\left( -1-\frac{2a_{12}}{a_{22}}-\frac{a_{12}^{2}}{%
a_{22}^{2}}+\frac{1}{a_{22}}\right) .
\]%
Let us substitute
\[
  Y_n=\left( n-\frac{1}{n}\right)\sqrt{a_{12}}\quad \text{i}\quad T_n=n.
\]
Thus
\[
  T_n-\frac{Y_n}{\sqrt{a_{12}}}=\frac{1}{n}
\]
and
\[
  \widetilde{p}_n\left( T_n, T_n-\frac{Y_n}{\sqrt{a_{12}}}\right)=\frac{1}{n^2}.
\]
Hence
\begin{align*}
  Q_1(T_n,Y_n)&=\widetilde{p}_4\left( T_n,T_n-\frac{Y_n}{\sqrt{a_{12}}}\right)+
    \left( \frac{a_{22}}{a_{12}}+3+\frac{a_{12}}{a_{22}}-\frac{1}{a_{12}} - \frac{6a_{12}}{a_{11}a_{22}-a_{12}^2}\right) n^2+ \\[.2cm]
    &+\left( \frac{2}{a_{22}}-\frac{1}{a_{12}}-\frac{1}{a_{11}a_{22}-a_{12}^2}\right) a_{12}\left(n-\frac{1}{n}\right)^2 +\\[.2cm]
    & +
    \left( \frac{6\sqrt{a_{12}}}{a_{11}a_{22}-a_{12}^2} -\frac{2}{\sqrt{a_{12}}} -\frac{4\sqrt{a_{12}}}{a_{22}}\right) \sqrt{a_{12}}\left(n^2-1\right)+ \\[.2cm] &+
    2-\frac{2}{a_{22}}+\frac{1}{a_{11}a_{22}-a_{12}^2} \left(1-a_{22} -2a_{12}-\frac{a_{12}^2}{a_{22}}\right) +
    \frac{3a_{12}^2}{(a_{11}a_{22}-a_{12}^2)^2}.
\end{align*}
Then the coefficient $B_{n^{2}}$ of the variable $n^2$  is equal
to
\[
B_{n^{2}}=\frac{a_{22}}{a_{12}}-\frac{a_{12}}{a_{22}}-\frac{1}{a_{12}}-\frac{%
a_{12}}{a_{11}a_{22}-a_{12}^{2}}.
\]%
We will show that $B_{n^{2}}$ is negative.

Let us assume first that $B_{n^2}$ is not negative. Then
\[
\frac{a_{22}}{a_{12}}\geqslant\frac{a_{12}}{a_{22}}+\frac{1}{a_{12}}+\frac{a_{12}}{%
a_{11}a_{22}-a_{12}^{2}}.
\]%
Successive calculation gives
\[
a_{22}^{2}\geqslant
a_{22}+a_{12}^{2}+\frac{a_{12}^{2}a_{22}}{a_{11}a_{22}-a_{12}^{2}}
\]
(because $a_{12}>0$),
\begin{equation} \label{gw}
a_{12}^{2}\left( a_{12}^{2}-a_{22}^{2}\right)\geqslant
a_{11}a_{22}\left( a_{22}+a_{12}^{2}-a_{22}^{2}\right)
\end{equation}

Since $a_{22}+a_{12}^{2}-a_{22}^{2}=a_{22}(1-a_{22})+a_{12}^{2}>0$
and $a_{11}a_{22}>a_{12}^2$ we have
\[
a_{11}a_{22}\left( a_{22}+a_{12}^{2}-a_{22}^{2}\right)
>a_{12}^{2}\left(a_{22}+a_{12}^{2}-a_{22}^{2}\right).
\]%
Then by (\ref{gw}) we obtain
\[
a_{12}^{2}-a_{22}^{2}\geqslant a_{22}+a_{12}^{2}-a_{22}^{2},
\]%
and therefore
\[
0\geqslant a_{22}.
\]%
The last inequality is false, because $a_{22}>0.$ Then $B_{n^2}<0$
and there  exists such $n\in\mathbb{N},$ that $Q_1(T_n,Y_n)<0.$ It
means that $f_1(x_{1,n},x_{2,n})<0$ at $(x_{1,n},x_{2,n})$
corresponding to $(T_n,Y_n)$ -- contradiction.

When $a_{12}<0$ we shall substitute in the above considerations
$a_{12}$ by $-a_{12}$. We obtain the same contradiction which
means that the function $f_{1}$ is not a density function and
$\varphi$ is indecomposable.

Then the decomposition theorem, which is true in one-dimensional
case is not true in d-dimensional case, where $d>1$.}
\end{example}

In the next part of this paper we will show that if some condition
holds for the associated polynomial of characteristic function
$\varphi$ of d-dimensional polynomial-normal distribution which
has no zeros then the characteristic function is decomposable. Let
us first prove some auxiliary propositions.

\begin{proposition}
Let
\[
Q(\mathbf{x})=\sum\limits_{\left\vert \alpha \right\vert \leq 2m}a_{\alpha }%
\mathbf{x}^{\alpha }
\]%
(where $\alpha$ denotes a multiindex,
$\alpha=(\alpha_1,\dots,\alpha_d)\in \mathbb{Z}_{+}^{d}$ and
$\left\vert \alpha \right\vert = \alpha_1+\ldots+\alpha_d$) be the
polynomial of $d$ variables, of degree $2m$ with real
coefficients, positive on  $\mathbb{R}^{d}.$ Let $Q$ satisfies the
following condition%

\begin{equation}
a:=\inf_{\mathbf{x}\in {\mathbb{R}}^{d}}
\frac{Q\left( x_{1},...,x_{d}\right) }{1+\sum\limits_{\left\vert \alpha \right\vert \leq 2m} \left\vert \mathbf{x}^{\alpha }\right\vert }>0. \label{3.28} 
\end{equation}%
Then there exists $\varepsilon >0$ such that, if%
\[
W(\mathbf{x})=\sum\limits_{\left\vert \alpha \right\vert \leq 2m}b_{\alpha }%
\mathbf{x}^{\alpha },\text{ \ \ \ }\mathbf{x}\in \mathbb{R}^{d}
\]%
and for every $\alpha $
\[
\left\vert a_{\alpha }-b_{\alpha }\right\vert <\varepsilon ,
\]%
then the polynomial $W$ has only positive values on $\mathbb{R}^{d}$.
\end{proposition}

\begin{proof}
If $\mathbf{x}=\left( x_{1},...,x_{d}\right) \in \mathbb{R}^{d},$ then%
\begin{equation}
W(\mathbf{x})=\sum\limits_{\left\vert \alpha \right\vert \leq 2m}b_{\alpha }%
\mathbf{x}^{\alpha }=\sum\limits_{\left\vert \alpha \right\vert \leq
2m}\left( b_{\alpha }-a_{\alpha }\right) \mathbf{x}^{\alpha
}+\sum\limits_{\left\vert \alpha \right\vert \leq 2m}a_{\alpha }\mathbf{x}%
^{\alpha }\geq \label{3.29} 
\end{equation}%
\[
\geq Q(\mathbf{x})-\sum\limits_{\left\vert \alpha \right\vert \leq
2m}\left\vert b_{\alpha }-a_{\alpha }\right\vert \left\vert \mathbf{x}%
^{\alpha }\right\vert \geq Q(\mathbf{x})-\left( \sum\limits_{\left\vert
\alpha \right\vert \leq 2m}\varepsilon \left\vert \mathbf{x}^{\alpha
}\right\vert +\varepsilon \right) .
\]%
From (\ref{3.28}) we have%
\[
\frac{Q\left( \mathbf{x}\right) }{1+\sum\limits_{\left\vert \alpha
\right\vert \leq 2m}\left\vert \mathbf{x}^{\alpha }\right\vert }\geq a,
\]%
so%
\begin{equation}
Q\left( \mathbf{x}\right) \geq a\left( 1+\sum\limits_{\left\vert \alpha
\right\vert \leq 2m}\left\vert \mathbf{x}^{\alpha }\right\vert \right) .
 \label{3.30} 
\end{equation}%
From (\ref{3.29}) and (\ref{3.30}) we have%
\[
W(\mathbf{x})\geq a\left( 1+\sum\limits_{\left\vert \alpha \right\vert \leq
2m}\left\vert \mathbf{x}^{\alpha }\right\vert \right) -\varepsilon \left(
1+\sum\limits_{\left\vert \alpha \right\vert \leq 2m}\left\vert \mathbf{x}%
^{\alpha }\right\vert \right)
=\left( a-\varepsilon \right) \left( 1+\sum\limits_{\left\vert \alpha
\right\vert \leq 2m}\left\vert \mathbf{x}^{\alpha }\right\vert \right) .
\]%
If $\varepsilon <a,$ then $W(\mathbf{x})>0,$ and the theorem is proved.

\end{proof}

Let us see that for every $\mathbf{x}\in \mathbb{R}^{d}$ the following equalities and inequalities hold

\begin{equation}
1+\sum\limits_{j=1}^{d}\left\vert x_{j}\right\vert ^{2m}\leq
1+\sum\limits_{\left\vert \alpha \right\vert \leq 2m}\left\vert \mathbf{x}%
^{\alpha }\right\vert  \label{3.31} 
\end{equation}%
(the left hand side is the addend of the right hand side);
\[
1+\sum\limits_{\left\vert \alpha \right\vert \leq 2m}\left\vert \mathbf{x}%
^{\alpha }\right\vert =1+\sum\limits_{\left\vert \alpha \right\vert \leq
2m}\left\vert x_{1}\right\vert ^{\alpha _{1}}..\left\vert x_{d}\right\vert
^{\alpha _{d}}\leq 1+\sum\limits_{\left\vert \alpha \right\vert \leq
2m}\left( 1+\sum\limits_{j=1}^{d}\left\vert x_{j}\right\vert \right)
^{\alpha _{1}}...\left( 1+\sum\limits_{j=1}^{d}\left\vert x_{j}\right\vert
\right) ^{\alpha _{d}}\leq
\]%
\begin{equation}
\leq \left( 1+\sum\limits_{j=1}^{d}\left\vert x_{j}\right\vert \right)
^{2m}+\sum\limits_{\left\vert \alpha \right\vert \leq 2m}\left(
1+\sum\limits_{j=1}^{d}\left\vert x_{j}\right\vert \right) ^{2m}=\left(
1+E_{2m}^{d}\right) \left( 1+\sum\limits_{j=1}^{d}\left\vert
x_{j}\right\vert \right) ^{2m}, \label{3.32} 
\end{equation}%
where $E_{2m}^{d}$ is the number of elements of the set $\left\{
\alpha \in
\mathbb{Z}_{+}^{d}:\left\vert \alpha \right\vert \leq 2m\right\} ;$%
\[
\left\{
\begin{array}{c}
1\leq \sqrt[2m]{1+\sum\limits_{j=1}^{d}\left\vert x_{j}\right\vert ^{2m}},
\\
\left\vert x_{j}\right\vert \leq \sqrt[2m]{1+\sum\limits_{k=1}^{d}\left%
\vert x_{k}\right\vert ^{2m}}\text{ \ for \ }j=1,2,\ldots,d%
\end{array}%
\right.
\]%
and
\begin{align}
\left( 1+\sum\limits_{j=1}^{d}\left\vert x_{j}\right\vert \right) ^{2m} & \leq
\left( \sqrt[2m]{1+\sum\limits_{j=1}^{d}\left\vert x_{j}\right\vert ^{2m}}%
+\sum\limits_{j=1}^{d}\sqrt[2m]{1+\sum\limits_{k=1}^{d}\left\vert
x_{k}\right\vert ^{2m}}\right) ^{2m} =\nonumber \\ & =\left(
1+d\right) ^{2m}\left( 1+\sum\limits_{j=1}^{d}\left\vert
x_{j}\right\vert ^{2m}\right) .
\label{3.33} 
\end{align}%
From (\ref{3.31}), (\ref{3.32}) and (\ref{3.33}) we have
\[
1+\sum\limits_{j=1}^{d}\left\vert x_{j}\right\vert ^{2m}\leq
1+\sum\limits_{\left\vert \alpha \right\vert \leq 2m}\left\vert \mathbf{x}%
^{\alpha }\right\vert \leq \left( 1+E_{2m}^{d}\right) \left(
1+\sum\limits_{j=1}^{d}\left\vert x_{j}\right\vert \right) ^{2m}\leq
\]%
\begin{equation*}
\leq \left( 1+E_{2m}^{d}\right) \left( 1+d\right) ^{2m}\left(
1+\sum\limits_{j=1}^{d}\left\vert x_{j}\right\vert ^{2m}\right). 
\end{equation*}%
Then we can replace the condition in the above proposition to the
one of the following equivalent conditions
\begin{equation}
b:=\inf_{\mathbf{x}\in \mathbb{R}^{d}}\frac{Q\left( x_{1},...,x_{d}\right) }{%
1+\sum\limits_{j=1}^{d}\left\vert x_{j}\right\vert ^{2m}}>0 ,\label{3.35(RW1)} 
\end{equation}%
\begin{equation}
c:=\inf_{\mathbf{x}\in \mathbb{R}^{d}}\frac{Q\left( x_{1},...,x_{d}\right) }{%
\left( 1+\sum\limits_{j=1}^{d}\left\vert x_{j}\right\vert \right) ^{2m}}>0 .
\label{3.36(RW2)} 
\end{equation}

\begin{proposition}
Let $Q(\mathbf{x})=\sum\limits_{\left\vert \alpha \right\vert
\leq 2m}a_{\alpha }\mathbf{x}^{\alpha }$ be a polynomial of degree $2m$ with real coefficients, positive on
 $\mathbb{R}^{d}.$ Then (\ref{3.28}) is equivalent to the following condition %

\begin{equation}
\text{for every }\ 1\leq j\leq d\qquad\text{we have }\quad \widetilde{a}_{j,2m}>0 ,
\label{3.37(RW)} 
\end{equation}%
where $\widetilde{a}_{j,2m}$ is the coefficient of $x_{j}^{2m}$ in
$Q$.

\end{proposition}

\begin{proof}
If $e_{j}$ is the $j$-th vector of the standard base in
$\mathbb{R}^{d}$
 \ $e_{j}=\left( \delta _{j1},\delta _{j2},...,\delta _{jd}\right) ,$
where
\[
\delta _{ji}=\left\{
\begin{array}{c}
1\quad\text{for \ }i=j, \\
0\quad\text{for \ }i\neq j,%
\end{array}%
\right.
\]%
is the Kronecker's symbol, then for every $r\in \left\{
0,1,...,2m\right\} $ the sequence $re_{j}\in \mathbb{Z}_+^d$ is
$d-$index of length $\left\vert re_{j}\right\vert \leq 2m.$ Let us
denote
\[
\widetilde{a}_{j,r}:=a_{re_{j}}.
\]

$(\ref{3.28})\Longrightarrow \left(\ref{3.37(RW)}\right) .$ Let us see that for every
 $1\leq j\leq d$ polynomial of one variable
\[
Q\left( x_{j}e_{j}\right) =\sum\limits_{r=0}^{2m}\widetilde{a}%
_{j,r}x_{j}^{r},\text{ \ \ }x_{j}\in \mathbb{R},
\]%
has positive values on $\mathbb{R}.$ Then $%
a_{j,0}=a_{0}=Q(0)>0$ and $\widetilde{a}_{j,2m}\geq 0.$

Let $\widetilde{a}_{j,2m}=0.$ Then
\[
\lim_{x_{j}\rightarrow +\infty }\frac{Q\left( x_{j}e_{j}\right) }{%
1+\left\vert x_{j}\right\vert ^{2m}}=\lim_{x_{j}\rightarrow
+\infty
}\sum\limits_{r=0}^{2m-1}\frac{\widetilde{a}_{j,r}x_{j}^{r}}{1+\left\vert
x_{j}\right\vert ^{2m}}=0
\]%
and
\[
\inf_{\mathbf{x}\in\mathbb{R}^{d}}\frac{Q\left( \mathbf{x}\right) }{%
1+\sum\limits_{j=1}^{d}\left\vert x_{j}\right\vert ^{2m}}=0.
\]%
The last equality contradicts the condition (\ref{3.35(RW1)}), so
it contradicts (\ref{3.28}). Then we must have
$\widetilde{a}_{j,2m}>0.$

$(\ref{3.28})\Longleftarrow \left(\ref{3.37(RW)}\right) .$ Let us write
 the polynomial $Q$ in the form%
\[
Q(\mathbf{x)=}\sum\limits_{j=1}^{d}\widetilde{a}_{j,2m}x_{j}^{2m}+\sum%
\limits_{\left\vert \alpha \right\vert <2m}a_{\alpha }\mathbf{x}^{\alpha },%
\text{ \ \ \ \ }\mathbf{x=}\left( x_{1},...,x_{d}\right) \in \mathbb{R}^{d},
\]%
and let us denote \ $\left\vert \left\vert \mathbf{x}\right\vert \right\vert
_{1}=\sum\limits_{j=1}^{d}\left\vert x_{j}\right\vert $ ($l^{1}-$norm in $%
\mathbb{R}^{d}$). From consideration concerning inequality
(\ref{3.32}) follows that for every $\alpha \in
\mathbb{Z}_{+}^{d}$ such that $\left\vert \alpha \right\vert =r$
the following equality holds
\[
\frac{\left\vert \mathbf{x}^{\alpha }\right\vert }{\left(
1+\sum\limits_{j=1}^{d}\left\vert x_{j}\right\vert \right) ^{2m}}\leq \frac{%
\left( 1+\sum\limits_{j=1}^{d}\left\vert x_{j}\right\vert \right) ^{r}}{%
\left( 1+\sum\limits_{j=1}^{d}\left\vert x_{j}\right\vert \right) ^{2m}}=%
\frac{1}{\left( 1+\left\vert \left\vert \mathbf{x}\right\vert \right\vert
_{1}\right) ^{2m-r}}.
\]%
Then taking (\ref{3.33}) we have
\[
\frac{\left\vert \sum\limits_{\left\vert \alpha \right\vert <2m}a_{\alpha }%
\mathbf{x}^{\alpha }\right\vert }{1+\sum\limits_{j=1}^{d}\left\vert
x_{j}\right\vert ^{2m}}\leq \frac{\left( 1+d\right) ^{2m}\left\vert
\sum\limits_{\left\vert \alpha \right\vert <2m}a_{\alpha }\mathbf{x}%
^{\alpha }\right\vert }{\left( 1+\left\vert \left\vert \mathbf{x}\right\vert
\right\vert _{1}\right) ^{2m}}\leq \sum\limits_{\left\vert \alpha
\right\vert <2m}\frac{\left\vert a_{\alpha }\right\vert }{\left(
1+\left\vert \left\vert \mathbf{x}\right\vert \right\vert _{1}\right)
^{2m-\left\vert \alpha \right\vert }}\left( 1+d\right) ^{2m}
\]%
and consistently
\begin{equation*}
\lim_{\left\vert \left\vert \mathbf{x}\right\vert \right\vert
_{1}\rightarrow +\infty }\frac{\sum\limits_{\left\vert \alpha \right\vert
<2m}a_{\alpha }\mathbf{x}^{\alpha }}{1+\sum\limits_{j=1}^{d}\left\vert
x_{j}\right\vert ^{2m}}=\lim_{\left\vert \left\vert \mathbf{x}\right\vert
\right\vert _{1}\rightarrow +\infty }\frac{\sum\limits_{\left\vert \alpha
\right\vert <2m}a_{\alpha }\mathbf{x}^{\alpha }}{\left( 1+\left\vert
\left\vert \mathbf{x}\right\vert \right\vert _{1}\right) ^{2m}}=0.
\end{equation*}
Let us denote
\[
b_{0}:=\min_{1\leq j\leq d}\widetilde{a}_{j,2m}.
\]%
Then
\[
  \begin{CD}
  \displaystyle \frac{\sum\limits_{j=1}^{d}\widetilde{a}_{j,2m}x_{j}^{2m}}{%
1+\sum\limits_{j=1}^{d}x_{j}^{2m}}\geq \frac{b_{0}\sum%
\limits_{j=1}^{d}x_{j}^{2m}}{1+\sum\limits_{j=1}^{d}x_{j}^{2m}}
    @>>{\left\vert \left\vert \mathbf{x}\right\vert \right\vert
_{1}\rightarrow +\infty}> b_0
  \end{CD}
\]
(Because in $\mathbb{R}^{d}$ $l^{1}-$norm $\left\vert \left\vert
\mathbf{\cdot }\right\vert \right\vert _{1}$ is equivalent $%
l^{2m}- $norm $\left\vert \left\vert \mathbf{\cdot }\right\vert
\right\vert _{2m}, $ so $\sum\limits_{j=1}^{d}x_{j}^{2m}=\left(
\left\vert \left\vert \mathbf{x}\right\vert \right\vert _{2m}\right)
^{2m}\rightarrow \infty ,$ when $\left\vert \left\vert \mathbf{x}\right\vert
\right\vert _{1}\rightarrow \infty $).

Let us take $R>0$ such that for $\left\vert \left\vert \mathbf{x}%
\right\vert \right\vert _{1}>R$ the following inequalities hold%
\[
\left\vert \frac{\sum\limits_{\left\vert \alpha \right\vert <2m}a_{\alpha }%
\mathbf{x}^{\alpha }}{1+\sum\limits_{j=1}^{d}x_{j}^{2m}}\right\vert <\frac{%
b_{0}}{3},
\]%

\[
\frac{\sum\limits_{j=1}^{d}\widetilde{a}_{j,2m}x_{j}^{2m}}{%
1+\sum\limits_{j=1}^{d}x_{j}^{2m}}\geq \frac{2}{3}b_{0}.
\]%
Let us denote by $\overline{K_{1}}(0,R)$ the closed ball in $%
\mathbb{R}^{d}$ with center at $\mathbf{x=0}$ and radius $R$ with respect to the
metric given by the norm $\left\vert \left\vert \mathbf{\cdot }%
\right\vert \right\vert _{1}.$ Then for $\mathbf{x}\in \mathbb{R}^{d}\setminus \overline{K_{1}}(0,R)$
\[
\frac{Q(\mathbf{x})}{1+\sum\limits_{j=1}^{d}x_{j}^{2m}}=\frac{%
\sum\limits_{j=1}^{d}\widetilde{a}_{j,2m}x_{j}^{2m}}{1+\sum%
\limits_{j=1}^{d}x_{j}^{2m}}+\frac{\sum\limits_{\left\vert \alpha
\right\vert <2m}a_{\alpha }\mathbf{x}^{\alpha }}{1+\sum%
\limits_{j=1}^{d}x_{j}^{2m}}\geq \frac{\sum\limits_{j=1}^{d}\widetilde{a}%
_{j,2m}x_{j}^{2m}}{1+\sum\limits_{j=1}^{d}x_{j}^{2m}}-\left\vert \frac{%
\sum\limits_{\left\vert \alpha \right\vert <2m}a_{\alpha }\mathbf{x}%
^{\alpha }}{1+\sum\limits_{j=1}^{d}x_{j}^{2m}}\right\vert >\frac{b_{0}}{3}.
\]%
We have%
\[
\inf_{\mathbf{x}\in \mathbb{R}^{d}}\frac{Q(\mathbf{x})}{1+\sum%
\limits_{j=1}^{d}x_{j}^{2m}}\geq \min \left\{ \frac{b_{0}}{3},\inf_{\mathbf{%
x}\in \overline{K_{1}}(0,R)}\frac{Q(\mathbf{x})}{1+\sum%
\limits_{j=1}^{d}x_{j}^{2m}}\right\} >0,
\]%
because the greatest lower bound of the continuous function%
\[
h(\mathbf{x)=}\frac{Q(\mathbf{x})}{1+\sum\limits_{j=1}^{d}x_{j}^{2m}},\text{
\ \ }\mathbf{x}\in \mathbb{R}^{d},
\]%
on the compact set $\overline{K_{1}}(0,R)$ is the value of the
function at some point
 $\overline{\mathbf{x}}\in \overline{K_{1}}(0,R),$ so it is
a positive number. It means that the condition (\ref{3.35(RW1)})
holds, so (\ref{3.37(RW)}) is fulfilled.
\end{proof}

Finally from Proposition 5 and Proposition 6 we have

\begin{theorem}
Let%
\[
Q(\mathbf{x})=\sum\limits_{\left\vert \alpha \right\vert \leq 2m}a_{\alpha }%
\mathbf{x}^{\alpha }
\]%
be the polynomial with real coefficients of degree $2m,$ positive
on $\mathbb{R}^{d}$ and such that the condition (\ref{3.37(RW)})
is satisfied. Then there exists
 $\varepsilon >0$ such that if%
\[
W(\mathbf{x)=}\sum\limits_{\left\vert \alpha \right\vert \leq 2m}b_{\alpha }%
\mathbf{x}^{\alpha },\text{ \ }\mathbf{x}\in \mathbb{R}^{d},
\]%
is the polynomial with real coefficients fulfilling inequalities

\[
\left\vert a_{\alpha }-b_{\alpha }\right\vert <\varepsilon ,
\]%
for every $\alpha$, $\vert{\alpha} \vert<2m, $ then the polynomial
$W$ takes only positive values on $\mathbb{R}^{d}.$
\end{theorem}

We show that property (\ref{3.37(RW)})
is invariant with respect to superposition of the polynomial $Q$
with nonsingular linear map of $\mathbb{R}^{d}.$

\begin{proposition}
If a positive polynomial $Q\in \mathbb{R}\left[ x_{1},...,x_{d}\right] $
of degree $2m$ satisfies condition (\ref{3.37(RW)}) and $\mathcal{A}$ is an
affine isomorphism of $\mathbb{R}^{d},$ then the polynomial $Q_{%
\mathcal{A}}:=Q\circ \mathcal{A}\in \mathbb{R}\left[ x_{1},...,x_{d}\right] $
is positive and satisfies condition (\ref{3.37(RW)}).
\end{proposition}

\begin{proof}
Let us remind that every affine isomorphism in $\mathbb{R}^{d}$ has the form%
\[
\mathcal{A(}\mathbf{x)=}F(\mathbf{x)+b,}\text{ \ }\mathbf{x}\in \mathbb{R}^{d},
\]%
where $F$ is a linear isomorphism and $\mathbf{b}\in \mathbb{R}^{d}.$

Positivity of the polynomial $Q_{\mathcal{A}}$ is clear. Let us
assume that
 $\mathbf{b=0}$ and denote by $\left\vert \left\vert F^{-1}\right\vert
\right\vert _{1}$ the operator norm of the map $F^{-1}$ with
respect to $l^{1}-$norm $\left\vert \left\vert \cdot \right\vert
\right\vert _{1}$ in $\mathbb{R}^{d}.$ Then $\left\vert \left\vert
F^{-1}\right\vert \right\vert _{1}>0.$ Since $Q$ satisfies
(\ref{3.37(RW)}) it satisfies (\ref{3.36(RW2)}), too. We will show
that $Q_{\mathcal{A}}$ satisfies (\ref{3.36(RW2)}). Let us
consider $\mathbf{x,y}\in \mathbb{R}^{d}$ such that
$\mathbf{x=}F\mathbf{(y)}.$ We have $\left\vert \left\vert
\mathbf{y}\right\vert \right\vert _{1}=\left\vert \left\vert
F^{-1}(\mathbf{x})\right\vert \right\vert _{1}\leq \left\vert
\left\vert F^{-1}\right\vert \right\vert _{1}\left\vert
\left\vert \mathbf{x}\right\vert \right\vert _{1}$ and then%
\[
\frac{Q_{\mathcal{A}}(\mathbf{y})}{\left( 1+\left\vert \left\vert \mathbf{y}%
\right\vert \right\vert _{1}\right) ^{2m}}
    =\frac{Q(\mathbf{x})}{\left(1+\left\vert \left\vert F^{-1}(\mathbf{x})\right\vert \right\vert
_{1}\right) ^{2m}}
    \geq \frac{Q(\mathbf{x})}{\left( 1+\left\vert \left\vert
F^{-1}\right\vert \right\vert _{1}\left\vert \left\vert \mathbf{x}%
\right\vert \right\vert _{1}\right) ^{2m}}.
\]%
There are two possibilities.

\textbf{Case 1}: $\left\vert \left\vert F^{-1}\right\vert \right\vert
_{1}>1.$ Then%
\[
1+\left\vert \left\vert F^{-1}\right\vert \right\vert _{1}\left\vert
\left\vert \mathbf{x}\right\vert \right\vert _{1}<\left\vert \left\vert
F^{-1}\right\vert \right\vert _{1}+\left\vert \left\vert F^{-1}\right\vert
\right\vert _{1}\left\vert \left\vert \mathbf{x}\right\vert \right\vert
_{1}=\left\vert \left\vert F^{-1}\right\vert \right\vert _{1}\left(
1+\left\vert \left\vert \mathbf{x}\right\vert \right\vert _{1}\right)
\]%
and%
\[
\frac{Q_{\mathcal{A}}(\mathbf{y})}{\left( 1+\left\vert \left\vert \mathbf{y}%
\right\vert \right\vert _{1}\right) ^{2m}}>\frac{Q(\mathbf{x})}{\left\vert
\left\vert F^{-1}\right\vert \right\vert _{1}^{2m}\left( 1+\left\vert
\left\vert \mathbf{x}\right\vert \right\vert _{1}\right) ^{2m}}\geq \frac{c}{%
\left\vert \left\vert F^{-1}\right\vert \right\vert _{1}^{2m}}>0,
\]%
where $c$ is the least lower bound from condition (\ref{3.36(RW2)}) for $Q.$ Hence%
\[
\inf_{\mathbf{y}\in \mathbb{R}^{d}}\frac{Q_{\mathcal{A}}(\mathbf{y})}{\left(
1+\left\vert \left\vert \mathbf{y}\right\vert \right\vert _{1}\right) ^{2m}}%
\geq \frac{c}{\left\vert \left\vert F^{-1}\right\vert \right\vert _{1}^{2m}}%
>0.
\]

\textbf{Case 2}: $\left\vert \left\vert F^{-1}\right\vert \right\vert
_{1}\leq 1.$ Then%
\[
1+\left\vert \left\vert F^{-1}\right\vert \right\vert _{1}\left\vert
\left\vert \mathbf{x}\right\vert \right\vert _{1}\leq 1+\left\vert
\left\vert \mathbf{x}\right\vert \right\vert _{1}
\]%
and%
\[
\frac{Q_{\mathcal{A}}(\mathbf{y})}{\left( 1+\left\vert \left\vert \mathbf{y}%
\right\vert \right\vert _{1}\right) ^{2m}}\geq \frac{Q (\mathbf{%
x})}{\left( 1+\left\vert \left\vert \mathbf{x}\right\vert \right\vert
_{1}\right) ^{2m}}\geq c>0.
\]%
This means that%
\[
\inf_{\mathbf{y}\in \mathbb{R}^{d}}\frac{Q_{\mathcal{A}}(\mathbf{y})}{\left(
1+\left\vert \left\vert \mathbf{y}\right\vert \right\vert _{1}\right) ^{2m}}%
\geq c>0.
\]%
In both cases the polynomial $Q_{\mathcal{A}}$ satisfies condition
(\ref{3.36(RW2)}), so it satisfies (\ref{3.37(RW)}), too.

Let us assume now that $F=id_{\mathbb{R}^{d}}$ which means that $%
\mathcal{A}$ is the translation by vector $\mathbf{b=}\left(
b_{1},...,b_{d}\right) .$ Then from Newton's formula for $%
\left( y_{i}+b_{i}\right) ^{2m}$ we have that for every $1\leq
j\leq d$ the coefficient of $y_{i}^{2m}$ in the polynomial
$Q_{\mathcal{A}}$ is the same as the coefficient of  $x_{i}^{2m}$
in the polynomial $Q$. It means that $Q_{\mathcal{A}}$ satisfies
condition (\ref{3.37(RW)}) iff $Q$ satisfies this condition.
Thesis of the proposition follows from the fact that every affine
isomorphism is a composition of a linear isomorphism and some
translation.
\end{proof}

\begin{conclusion}
In the above proposition we can substitute implication by equivalence: the polynomial
 $Q$ satisfies condition (\ref{3.37(RW)})
iff $Q_{\mathcal{A}}$ satisfies (\ref{3.37(RW)}).
\end{conclusion}

\begin{proof}
If $\mathcal{A}$ is an isomorphism then $\mathcal{A}^{-1}$ is an
isomorphism, too and $Q=\left( Q_{\mathcal{A}}\right) _{\mathcal{A}^{-1}}.$
\end{proof}

The following theorem is the main result of this paper.

\begin{theorem}
Let the function $f$ be the density on
$\mathbb{R}^{d}$of the form%
\[
f(\mathbf{x})=\frac{\sqrt{\det \mathbf{A}}}{\left( 2\pi \right) ^{\frac{d}{2}}}p_{2l}(%
\mathbf{x})\exp \left( -\frac{1}{2}\left( \mathbf{x-b}\right)^T \mathbf{A}\left(
\mathbf{x-b}\right)  \right) ,
\]%
where $p_{2l}$ is a positive polynomial of degree $2l$ satisfying
(\ref{3.37(RW)}), and $A$ is a symmetric and positive matrix of
dimension $d\times d$. Let $\varphi$ be the characteristic
function of this distribution. Then there are $d$-dimensional
random variables $Y$ and $Z$ -- the first with polynomial-normal
distribution and the second with the normal distribution such that

\[
\varphi =\varphi _{Y}\varphi _{Z},
\]%
where $\varphi _{Y}$ and $\varphi _{Z}$ are characteristic
functions of  $Y$ and $Z$ respectively.
\end{theorem}

\begin{proof}
Let $\widetilde{X}$ be a $d$-dimensional random variable with density%
 $f$. Then the random variable $X:=\widetilde{X}-\mathbf{b}$ has the density of the form%
\begin{equation*}
f_{X}(\mathbf{x})=\frac{\sqrt{\det \mathbf{A}}}{\left( 2\pi \right) ^{\frac{d}{2}}}%
p_{2l}(\mathbf{x+b})\exp \left( -\frac{1}{2}\mathbf{x}^T \mathbf{A}\mathbf{x}\right).
\label{3.39} 
\end{equation*}
The polynomial $p_{2l}$
is positive (from Theorem 8) and satisfies condition (\ref{3.37(RW)}). We have

\[
\varphi (\mathbf{t})=\varphi _{X}(\mathbf{t})\exp \left(
i\mathbf{b}\cdot\mathbf{t} \right) ,\text{ \ }\mathbf{t}\in
\mathbb{R}^{d}.
\]%
If $\varphi _{X}=\varphi _{Y}\varphi _{Z},$ where $Y$ and $Z$ %
are $d$-dimensional random variables which satisfies condition of the thesis.
Then for every $\mathbf{t}\in \mathbb{R}^{d}$%
\[
\varphi (\mathbf{t})=\exp \left( i\mathbf{b}\cdot\mathbf{t}
\right) \varphi _{Y}(\mathbf{t})\varphi _{Z}(\mathbf{t})=\varphi _{Y}(%
\mathbf{t})\left[ \exp \left( i\mathbf{b}\cdot\mathbf{t}
\right) \varphi _{Z}(\mathbf{t})\right] =\varphi _{Y}(\mathbf{t})\varphi _{%
\widetilde{Z}}(\mathbf{t}),
\]%
where $\widetilde{Z}=Z+\mathbf{b}$ is $d$-dimensional random
variable of normal distribution. Hence we may assume that the
density $f$ from the thesis of our theorem has the form
\[
f(\mathbf{x})=\frac{\sqrt{\det \mathbf{A}}}{\left( 2\pi \right) ^{\frac{d}{2}}}p_{2l}(%
\mathbf{x})\exp \left( -\frac{1}{2}\mathbf{x}^{T}\mathbf{Ax}\right) .
\]

Let L be the matrix which reduce $\mathbf{A}$ to a normalized
diagonal matrix (to the identity matrix).

Let $\mathbf{X}=L\mathbf{U.}$ Then $\mathbf{U}=L^{-1}\mathbf{X}$ \
and the density function of $\mathbf{U}$ is given by

\[
f_{\mathbf{U}}\left( \mathbf{u}\right) =f_{L^{-1}\mathbf{X}}\left( \mathbf{u}%
\right) =f_{\mathbf{X}}\left( L\mathbf{u}\right) \left\vert \det
L\right\vert =
\]%
\[
=\frac{\sqrt{\det \mathbf{A}\left( \det L\right) ^{2}}}{\left( 2\pi \right) ^{\frac{d%
}{2}}}p_{2l}(L\mathbf{u})\exp \left( -\frac{1}{2}\left( L\mathbf{u}\right)
^{T}\mathbf{A}\left( L\mathbf{u}\right) \right) =
\]%
\[
=\frac{\sqrt{\det \left( L^{T}\mathbf{A}L\right) }}{\left( 2\pi \right) ^{%
\frac{d}{2}}}\widetilde{p_{2l}}(\mathbf{u})\exp \left( -\frac{1}{2}\mathbf{u}%
^{T}(L^{T}\mathbf{A}L)\mathbf{u}\right) =
\]%
\[
=\frac{1}{\left( 2\pi \right) ^{\frac{d}{2}}}\widetilde{p_{2l}}(\mathbf{u}%
)\exp \left( -\frac{1}{2}\mathbf{u}^{T}\mathbf{u}\right).
\]

From Proposition 8 the polynomial $\widetilde{p_{2l}}(\mathbf{u})=p_{2l}(L%
\mathbf{u})$ is positive and satisfies condition (\ref{3.37(RW)}). If $%
\varphi _{U}=\varphi _{\widetilde{\mathbf{Y}}}\varphi _{\widetilde{\mathbf{Z%
}}}$, where $\widetilde{\mathbf{Y}}$ has a polynomial-normal
distribution and
 $\widetilde{\mathbf{Z}}$ has a normal distribution then
\[
\varphi _{\mathbf{X}}\left( \mathbf{t}\right) =\varphi _{L\mathbf{U}}\left(
\mathbf{t}\right) =\varphi _{\mathbf{U}}\left( L^{\ast }\mathbf{t}\right)
=\varphi _{\widetilde{\mathbf{Y}}}\left( L^{\ast }\mathbf{t}\right) \varphi
_{\widetilde{\mathbf{Z}}}\left( L^{\ast }\mathbf{t}\right) =\varphi _{L%
\widetilde{\mathbf{Y}}}\left( \mathbf{t}\right) \varphi _{L\widetilde{%
\mathbf{Z}}}\left( \mathbf{t}\right) ,\text{ \ \ \ }\mathbf{t}\in \mathbb{R}%
^{d}.
\]%
We know that $\mathbf{Y}:=L\widetilde{\mathbf{Y}}$ has a
polynomial-normal distribution and
$\mathbf{Z}:=L\widetilde{\mathbf{Z}}$ has a normal distribution.
Hence, if the thesis of the theorem holds for the identity matrix
$I$ then it holds for any symmetric and positive matrix $%
\mathbf{A}$.

We must only consider the case when the density from the thesis is
given by
\[
f(\mathbf{x})=\frac{1}{\left( 2\pi \right) ^{\frac{d}{2}}}p_{2l}(\mathbf{x}%
)\exp \left( -\frac{1}{2}\sum\limits_{i=1}^{d}x_{i}^{2}\right) .
\]%
Let $\varphi$ be the characteristic function of $f$. Then it can
be written in the form
\[
\varphi \left( \mathbf{t}\right) =\sum\limits_{\left\vert \alpha
\right\vert \leq 2l}\beta _{\alpha }\left( i\mathbf{t}\right) ^{\alpha }\exp
\left( -\frac{1}{2}\sum\limits_{i=1}^{d}t_{i}^{2}\right) ,\text{ \ \ \ }%
\mathbf{t}\in \mathbb{R}^{d},
\]%
  where $\beta _{\alpha }\in\mathbb{R}$ (see \cite{lukacz3}, \S 7.3). Let
\begin{equation*}
\varphi _{\theta }\left( \mathbf{t}\right)
=\sum\limits_{\left\vert \alpha \right\vert \leq 2l}\beta _{\alpha
}\left( i\mathbf{t}\right) ^{\alpha }\exp \left(
-\frac{1}{2}\sum\limits_{i=1}^{d}\theta ^{2}t_{i}^{2}\right)
,\text{ \ \ \ }\mathbf{t}\in \mathbb{R}^{d},  \label{Lew}
\end{equation*}
where $\theta \in (0,1).$ Then the inverse Fourier transform of
$\varphi _{\theta }$ is equal to
\[
f_{\theta }(\mathbf{x})=\frac{1}{\left( 2\pi \right) ^{\frac{d}{2}}}\left[
\sum\limits_{\left\vert \alpha \right\vert \leq 2l}\beta _{\alpha }\frac{1}{%
\theta ^{\left\vert \alpha \right\vert }}\prod\limits_{j=1}^{d}H_{\alpha
_{j}}\left( \frac{x_{j}}{\theta }\right) \right] \exp \left( -\frac{1}{%
2\theta ^{2}}\sum\limits_{j=1}^{d}x_{j}^{2}\right) ,
\]%
where ${\alpha_{j}}$ denotes the Hermite polynomial of order
$\alpha_{j}$ (see \cite{lukacz3}, \S 7.3).
It is clear that%
\[
\lim_{\theta \rightarrow 1^{-}}f_{\theta }(\mathbf{x})=\frac{1}{\left( 2\pi
\right) ^{\frac{d}{2}}}\left[ \sum\limits_{\left\vert \alpha \right\vert
\leq 2l}\beta _{\alpha }\prod\limits_{j=1}^{d}H_{\alpha _{j}}\left(
x_{j}\right) \right] \exp \left( -\frac{1}{2}\sum\limits_{j=1}^{d}x_{j}^{2}%
\right) =f(\mathbf{x}).
\]%
Moreover the coefficients of the polynomial
\[
p_{2l,\theta }(\mathbf{x})=\sum\limits_{\left\vert \alpha \right\vert \leq
2l}\beta _{\alpha }\frac{1}{\theta ^{\left\vert \alpha \right\vert }}%
\prod\limits_{j=1}^{d}H_{\alpha _{j}}\left( \frac{x_{j}}{\theta }\right) ,%
\text{ \ \ \ }\mathbf{x}\in \mathbb{R}^{d},
\]%
converge to the coefficients of the polynomial%
\[
\sum\limits_{\left\vert \alpha \right\vert \leq 2l}\beta _{\alpha
}\prod\limits_{j=1}^{d}H_{\alpha _{j}}\left( x_{j}\right) =p_{2l}(\mathbf{x}%
),\text{ \ \ \ \ }\mathbf{x}\in \mathbb{R}^{d}.
\]%

Thus, if we write $p_{2l}$ in the form%
\[
p_{2l}(\mathbf{x})=\sum\limits_{\left\vert \alpha \right\vert \leq
2l}a_{\alpha }\mathbf{x}^{\alpha }
\]%
and%
\[
p_{2l,\theta }(\mathbf{x})=\sum\limits_{\left\vert \alpha \right\vert \leq
2l}b_{\alpha }(\theta )\mathbf{x}^{\alpha },
\]%
then for every $\alpha \in   \mathbb{Z}_{+}^{d}$ ,
 $\left\vert \alpha \right\vert \leq 2l$ we have%
\[
\lim_{\theta \rightarrow 1^{-}}b_{\alpha }(\theta )=a_{\alpha }.
\]%
Let $\varepsilon >0$ be a number from Theorem 7 chosen for the
polynomial $Q=p_{2l}.$ Then
there exists $\delta >0$ such that if $1-\theta <\delta ,$ then $%
\left\vert b_{\alpha }(\theta )-a_{\alpha }\right\vert
<\varepsilon $ for
 $\left\vert \alpha \right\vert \leq 2l.$ If we take  $\theta
\in \left( 1-\delta ,1\right) $  then we get nonnegative polynomial $%
p_{2l,\theta }$ on $\mathbb{R}^{d}.$ Hence $%
f_{\theta }$ is a density function on $\mathbb{R}%
^{d}$ and the random variable $\mathbf{Y}$ has $PND_{d}$ distribution.

Let%
\[
f_{\mathbf{Z}}(\mathbf{x})=\frac{1}{\left( 2\pi \right) ^{\frac{d}{2}}\left(
1-\theta ^{2}\right) ^{\frac{d}{2}}}\exp \left( -\frac{1}{2\left( 1-\theta
^{2}\right) }\sum\limits_{i=1}^{d}x_{i}^{2}\right).
\]%
Then
\[
\varphi _{\mathbf{Z}}\left( \mathbf{t}\right) =\exp \left( -\frac{1}{2}%
\left( 1-\theta ^{2}\right) \sum\limits_{i=1}^{d}t_{i}^{2}\right) .
\]%
and%
\[
\varphi _{\mathbf{Y}}\left( \mathbf{t}\right) \varphi _{\mathbf{Z}}\left(
\mathbf{t}\right) =\varphi _{\mathbf{\theta }}\left( \mathbf{t}\right)
\varphi _{\mathbf{Z}}\left( \mathbf{t}\right) =\varphi \left( \mathbf{t}%
\right) ,\text{ \ \ }\mathbf{t}\in \mathbb{R}^{d}.
\]%
This ends the proof of our theorem.
 \end{proof}


\textit{Acknowledgments.} We express our thanks to J. Weso\l owski
and Z. Jelonek for the inspiration and useful discussions during
writing this paper.

\end{document}